\pdfoutput=1

\documentclass[preprint,12pt,3p]{elsarticle}

\usepackage{graphicx}

\usepackage{amssymb, amsmath}
\usepackage{amsthm}
\usepackage{color}
\usepackage{enumitem}

\newtheorem{theorem}{Theorem}[section]
\newtheorem{lemma}[theorem]{Lemma}

\theoremstyle{definition}
\newtheorem{definition}[theorem]{Definition}

\newtheorem{example}[theorem]{Example}

\theoremstyle{remark}
\newtheorem{remark}[theorem]{Remark}

\begin{document}

\begin{frontmatter}

\title{On forced oscillations in groups\\ of interacting nonlinear systems}

\author{Ivan Polekhin}


\ead{ivanpolekhin@gmail.com}



\begin{abstract}
Consider a periodically forced nonlinear system which can be presented as a collection of smaller subsystems with pairwise interactions between them. Each subsystem is assumed to be a massive point moving with friction on a compact surface, possibly with a boundary, in an external periodic field. We present sufficient conditions for the existence of a periodic solution for the whole system. The result is illustrated by a series of examples including a chain of strongly coupled pendulums in a periodic field.
\end{abstract}


\end{frontmatter}


\section{Brief introduction}
\label{sec1}
The phenomenon of forced oscillations has been studied since at least 1922, when G.\,Hamel proved~\cite{hamel1922erzwungene} that the equation describing the motion of a periodically forced pendulum have at least one periodic solution. Various results, which generalize and develop~\cite{hamel1922erzwungene}, have appeared in the literature since this paper; see, for example,~\cite{ benci1990periodic, furi1995second, furi1991forced, furi1990existence, mawhin2000global}.  Forced oscillations in a system of coupled planar pendulums and its generalizations has been studied in~\cite{marlin1968periodic, mawhin1970g} under the assumption of certain symmetry properties for the forcing terms. The existence and multiplicity of periodic solutions for coupled systems also discussed in~\cite{drabek1987periodic}.

At the same time, a nonlinear lattice is a cornerstone model in nonlinear physics and it is widely used for analytical and computational purposes. See, for example,~\cite{kartashov2011solitons, toda2012theory, flach1998discrete, marin1996breathers}. In our work we present a result which lies in the intersection of the both mentioned research areas and can be useful for studying forced oscillations in a nonlinear lattice and its generalizations. The result given in the paper continues a previously reported result~\cite{polekhin2015na}.

 In the paper the following systems are considered. Let us have several compact smooth manifolds, possibly with boundaries, with non-zero Euler-Poincar\'{e} characteristics. Suppose that for each manifold there is a massive point moving on it with viscous friction. All points are in an external periodic field and may interact with each other. It is allowed that the interactions may be of different types and also may be arbitrarily strong. We present sufficient conditions for the existence of a periodic solution for such a system.

We would like to note that when the considered manifolds are closed, it isn't hard to prove that there exists a periodic solution in the system: it directly follows from the Lefschetz-Hopf theorem. Yet in applications it appears to be useful to consider manifolds with boundaries to prove the existence of a periodic solution and estimate it. One of the examples of such kind, which we consider further below, is a widely used in nonlinear studies model of coupled oscillators on a line. Therefore, our result generalizes the approach based on an application of the Lefschetz-Hopf theorem to the important case of compact manifolds with boundaries. 

Our result is based on a theorem by R.\,Srzednicki, K.\, W{\'o}jcik, and P.\,Zgliczy{\'n}ski~\cite{srzednicki2005fixed} and provides an illustrative geometrical approach to study periodical oscillations. Since the result is presented in a coordinate-free form, it also avoids possible shortcomings of purely analytical approaches and can be applied to the systems with complex topology of the phase space.

The main theorem is illustrated by a series of examples from mechanics including a system of an arbitrary number of strongly coupled pendulums in a periodic external field.

\section{Main result}
\label{sec2}

\subsection{Governing equations}
\label{subsec2_1}
The equations introduced in this subsection --- which we are going to use further below in the paper --- generalize the governing equations for a mechanical system consisting of massive points moving with friction-like interaction on compact surfaces. For the sake of simplicity, we assume that all manifolds and considered functions are smooth (i.e. $C^\infty$). 

Let $M_i$ be a compact connected one- or two-dimensional manifold (possibly with a boundary), where $i = 1,...,n$. For all $i=1,...,n$, there is a point moving in an external force field on the manifold $M_i$. We also assume that there is a friction-like force acting on the point.

In our further consideration, we will study the behavior of the system in vicinities of $\partial M_i$ and it will be convenient to consider enlarged manifolds $M^+_i$. Let $M^+_i$ be a boundaryless connected manifold such that $M_i \subset M^+_i$, $\dim M_i = \dim M_i^+$. Also suppose that every manifold $M_i^+$ is equipped with a Riemannian metric $\langle \cdot, \cdot \rangle_i$. 

If the points do not interact, one can consider the following independent equations of motion
\begin{equation}
\label{premain}
\nabla^i_{\dot q_i} \dot q_i = f_i(t, q_i, \dot q_i) + f_i^{friction}(t, q_i, \dot q_i), \quad i = 1,...,n,
\end{equation}
where $\nabla^i$ means the covariant differentiation with respect to the corresponding metric $\langle \cdot, \cdot \rangle_i$ for the $i$-th point moving on $M^+_i$; $f_i \colon \mathbb{R} / T\mathbb{Z} \times TM_i^+ \to TM_i^+$ and $f_i(t, q_i, \dot q_i) \in T_{q_i}M_i^+$ for any $t$, $q_i$, $\dot q_i$;  $f_i^{friction} \colon \mathbb{R} / T\mathbb{Z} \times TM_i^+ \to TM_i^+$ corresponds to a friction-like force.

As we said before,~(\ref{premain}) includes the case of a mechanical system of massive points in an external field. Indeed, in this case $\langle \cdot, \cdot \rangle_i$ is the corresponding kinetic metric and, for given $t$, $q_i$ and  $\dot q_i$, $f_i(t, q_i, \dot q_i)$ and $f_i^{friction}(t, q_i, \dot q_i)$ are the dual vectors to the corresponding generalized forces. Note that from~(\ref{premain}) it follows that if there is no forces acting on the $i$-th point then~(\ref{premain}) becomes $\nabla^i_{\dot q_i} \dot q_i = 0$, which is the equation of the geodesic motion.

In general form, when the interactions are took into account, we consider the following equations of motion
\begin{equation}
\label{FirstEq}
\nabla^i_{\dot q_i} \dot q_i = f_i(t, q_i, \dot q_i) + f_i^{friction}(t, q_i, \dot q_i) + f^{interaction}_i(t, q_1, \dot q_1, ..., q_n, \dot q_n), \quad i = 1,...,n,
\end{equation}
where $f^{interaction}_i \colon \mathbb{R} / T\mathbb{Z} \times TM_1^+ \times ... \times TM_n^+ \to TM_i^+$ and $f^{interaction}_i(t, q_1, \dot q_1, ..., q_n, \dot q_n) \in T_{q_i}M_i^+$ for any $t$, $q_j$, $\dot q_j$, $j = 1,...,n$. Here the functions $f^{interaction}_i$, being defined on $\mathbb{R} / T\mathbb{Z} \times TM_1^+ \times \dots \times TM_n^+$, may describe more general forces than the pairwise interactions usually considered in applications.

As usual, we say that the system has a $T$-periodic solution if there are functions $q_i~\colon~\mathbb{R}~\to~M_i^+$ which satisfy~(\ref{FirstEq}) and $q_i(t + T) = q_i(t)$, for all $t \in [0, T)$, $i = 1,...,n$.

Finally, rewrite~(\ref{FirstEq}) as follows
\begin{equation}
\label{MainEq}
\begin{aligned}
&\dot q_i = p_i,\\
&\nabla^i_{p_i} p_i = f_i(t, q_i, p_i) + f_i^{friction}(t, q_i, p_i) + f^{interaction}_i(t, q_1, p_1, ..., q_n, p_n),
\end{aligned}\quad i=1,...,n.
\end{equation}
In the rest of the paper, we will assume that $f_i^{friction}$, $f_i$ and $f_i^{interaction}$ satisfy the following conditions:
\begin{itemize}
\item[(H1)] There exist constants $d_i > 0$ such that 
\begin{equation}
\label{GammaEq}
\sup_{\substack{t \in [0, T],\, q_i \in M_i \\ \langle p_i, p_i \rangle_i > d_i}} \frac{\langle f_i^{friction}(t, q_i, p_i), p_i \rangle_i}{\langle p_i, p_i \rangle_i} < 0, \quad i=1,...,n.
\end{equation}
\item[(H2)] The functions $f_i$ and $f_i^{interaction}$ are bounded for all $i=1,...,n$.
\end{itemize}
\begin{remark}
If we consider a mechanical system of massive points moving on surfaces in $\mathbb{R}^3$ with viscous friction with arbitrarily small friction coefficients then~(\ref{GammaEq}) is satisfied.
\end{remark}
\begin{lemma}
\label{Lemma-I}
Suppose that conditions (H1) and (H2) are satisfied, then for some $c_i > 0$, $i=1,...,n$ along the solutions of~(\ref{MainEq})
\begin{equation*}
\frac{d}{dt} T_i \Big|_{T_i = c_i} < 0, \quad T_i = \langle p_i, p_i \rangle_i = \| p_i \|_i^2.
\end{equation*}
\end{lemma}
\begin{proof}
By direct calculation from~(\ref{MainEq}), we have
\begin{equation*}
\begin{aligned}
\frac{d}{dt} T_i &= 2\langle f_i(t, q_i, p_i) + f_i^{friction}(t, q_i, p_i) + f^{interaction}_i(t, q_1, p_1, ..., q_n, p_n),  p_i \rangle_i \leqslant \\
&\leqslant 2\langle p_i, p_i \rangle_i \left(\frac{\| f_i(t, q_i, p_i) \|_i}{\| p_i \|_i} + \frac{\| f^{interaction}_i(t, q_1, p_1, ..., q_n, p_n)\|_i}{\| p_i \|_i} + \frac{ \langle f_i^{friction}(t, q_i, p_i),  p_i \rangle_i}{\langle p_i, p_i \rangle_i}\right).
\end{aligned}
\end{equation*}
Since $f_i$ and $f_i^{interaction}$ are bounded, then from~(\ref{GammaEq}) we obtain $dT_i/dt < 0$ provided $c_i$ is large enough.
\end{proof}

\subsection{Auxiliary constructions and results}
\label{subsec2_2}

The approach developed in~\cite{srzednicki2005fixed} is based on the ideas of the Wa\.{z}ewski method~\cite{wazewski1947principe} and the Lefschetz-Hopf theorem. In this subsection we introduce some definitions and a result from~\cite{srzednicki2005fixed} which we slightly modify for our use.
\par
Let $v \colon \mathbb{R}\times M \to TM$ be a time-dependent vector field on a manifold $M$
\begin{equation}
\label{eq1}
\dot x = v(t, x).
\end{equation}
For $t_0 \in \mathbb{R}$ and $x_0 \in M$, the map $t \mapsto x(t,t_0,x_0)$ is the solution for the initial value problem for the system~(\ref{eq1}), such that $x(0,t_0,x_0)=x_0$. If $W \subset \mathbb{R}\times M$, $t\in\mathbb{R}$, then we denote
\begin{equation*}
W_t=\{x \in M \colon (t,x) \in W\}.
\end{equation*}
\begin{definition}
Let $W \subset \mathbb{R} \times M$. Define the {exit set} $W^-$ as follows. A point $(t_0,x_0)$ is in $W^-$ if there exists $\delta>0$ such that $(t+t_0, x(t,t_0,x_0)) \notin W$ for all $t \in (0,\delta)$.
\end{definition}
\begin{definition}
We call $W \subset \mathbb{R}\times M$ a {Wa\.{z}ewski block} for the system~(\ref{eq1}) if $W$ and $W^-$ are compact.
\end{definition}
\begin{definition}
A set $W \subset [a,b] \times M$ is called a simple periodic segment over $[a,b]$ if it is a Wa\.{z}ewski block with respect to the system~(\ref{eq1}), $W = [a, b] \times Z$, where $Z \subset M$, and $W^-_{t_1} = W^-_{t_2}$ for any $t_1, t_2 \in [a, b)$.
\end{definition}
\begin{definition}
Let $W$ be a simple periodic segment over $[a,b]$. The set $W^{--} = [a,b] \times W_a^-$ is called the essential exit set for $W$.
\end{definition}
In our case, the result from~\cite{srzednicki2005fixed} can be presented as follows.

\begin{theorem} 
\label{th1}
\cite{srzednicki2005fixed} Let W be a simple periodic segment over $[a,b]$. Then the set
\begin{equation*}
U = \{ x_0 \in W_a \colon x(t-a,a,x_0) \in W_t\setminus W_t^{--}\,\mbox{for all}\,\, t \in [a,b] \}
\end{equation*}
is open in $W_a$ and the set of fixed points of the restriction $x(b-a,a,\cdot)|_U \colon U \to W_a$ is compact. Moreover, if $W_a$ and $W^-_a$ are ANRs then the fixed point index of $x(b-a,a,\cdot)|_U$ can be calculated by means of the Euler-Poincar\'{e} characteristic of $W$ and $W^-_a$ as follows
\begin{equation*}
\mathrm{ind}(x(b-a,a,\cdot)|_U) = \chi(W_a) - \chi(W^-_a).
\end{equation*}
In particular, if $\chi(W_a) - \chi(W^-_a) \ne 0$ then $x(b-a,a,\cdot)|_U$ has a fixed point in $W_a$.
\end{theorem}
\newpage
\subsection{Main theorem}
\label{subsec2_3}

In this subsection, we prove our main result and illustrate it with a series of examples. Below we use the following notations: $M = M_1 \times ... \times M_n$ and $M^+ = M^+_1 \times ... \times M^+_n$.

\begin{theorem}
\label{MyMainTh}
Suppose that for~(\ref{MainEq}) the following conditions are satisfied
\begin{enumerate}
\item The Euler-Poincar\'{e} characteristic of $M_i$ is non-zero for all $i = 1,...,n$.
\item (H1) and (H2) are satisfied.
\item For any $t_0 \in \mathbb{R}$, $(q^0_1, p^0_1,...,q^0_n,p_n^0) \in T(\partial M)$ there is an $\varepsilon > 0$ such that
\begin{equation}
\label{Cond-IV}
q(t,t_0,q^0_1, p^0_1,...,q^0_n,p_n^0) \notin M, \quad \mbox{for all} \quad t \in (0, \varepsilon).
\end{equation}
\end{enumerate}
Then there exists a solution $(q_1, p_1, ..., q_n, p_n) \colon \mathbb{R} \to TM^+$ of~(\ref{MainEq}) such that for all $i=1,...,n$
\begin{equation*}
q_i(t) = q_i(t+T), \quad p_i(t) = p_i(t+T), \quad q_i(t) \in M_i \setminus \partial M_i, \quad\mbox{for all}\quad t \in \mathbb{R}.
\end{equation*}
\end{theorem}
\begin{proof}
Consider the following compact subset $W$ of~$[0,T]\times TM^+$
\begin{equation*}
W = \{ 0 \leqslant t \leqslant T, (q_1, p_1,...,q_n, p_n) \in TM^+ \colon (q_1, ...,  q_n) \in M_1 \times ... \times M_n, \langle p_i, p_i \rangle_i \leqslant c_i, i=1,...,n \}
\end{equation*}
where $c_i>0$ are the constants obtained from~lemma~\ref{Lemma-I}.

From~lemma~\ref{Lemma-I} we also have that if $(t,q_1,p_1,...,q_n,p_n) \in W^{--}$ then $(q_1,...,q_n) \in \partial M$. Since
\begin{equation}
\partial M = \partial M_1 \times M_2 \times ... \times M_n \cup ... \cup M_1 \times ... M_{n-1} \times \partial M_n,
\end{equation}
then for $(t,q_1,p_1,...,q_n,p_n) \in W^{--}$ there exists $i$ such that $q_i \in \partial M_i$.
Let $\nu_{q_i} \in T_{q_i} M_i^+$ be a normal vector to $\partial M_i$ at point $q_i \in \partial M_i$ such that, for all $p_i \in T_{q_i} M_i^+$, $\langle \nu_{q_i}, p_i \rangle_i > 0$, the solution starting from $(t,q_1,p_1,...,q_i,p_i,...,q_n,p_n)$ at least locally leaves $M$. From the above definition of $\nu_{q_i}$, we obtain for any $j$
\begin{equation*}
\begin{aligned}
W^{--} \supset \{  0 &\leqslant t \leqslant T, (q_1, p_1,...,q_n,p_n) \in TM^+ \colon\\
&(q_1,...,q_n) \in M, \langle p_i, p_i \rangle_i \leqslant c_i, i=1,...,n , q_j \in \partial M_j, \langle\nu_{q_j}, p_j\rangle_j > 0 \}.
\end{aligned}
\end{equation*}
Let us denote by $V^{--}_j$, $j=1,...,n$ the following set
\begin{equation*}
\begin{aligned}
V^{--}_j =\{  0 &\leqslant t \leqslant T, (q_1, p_1,...,q_n,p_n) \in TM^+ \colon\\
&(q_1,...,q_n) \in M,\langle p_i, p_i \rangle_i \leqslant c_i, i=1,...,n, \langle\nu_{q_j}, p_j\rangle_j \geqslant 0, q_j \in \partial M_j\}.
\end{aligned}
\end{equation*}
Since we assume~(\ref{Cond-IV}), then $W^{--}$ is compact:
\begin{equation}
\label{sum}
\begin{aligned}
W^{--} = V^{--}_1 \cup ... \cup V^{--}_n.
\end{aligned}
\end{equation}
\begin{figure}[h!]
\centering
\def\svgwidth{280 pt}
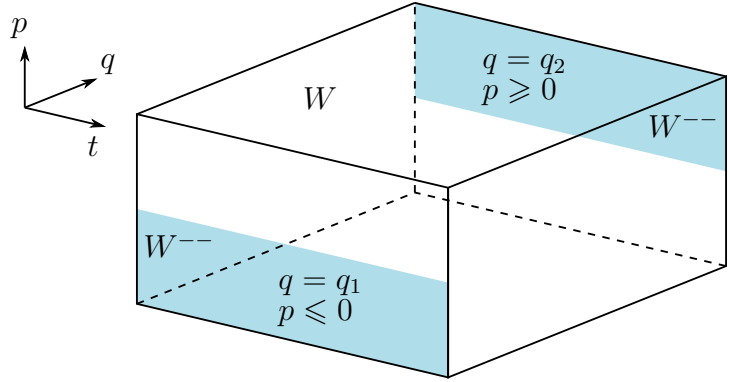
\caption{For a one-dimensional manifold parametrized by $q \in [q_1, q_2]$, $W$ has the above form.}
\label{fig5}
\end{figure}
Therefore, $W^{-}$ is also compact and $W$ is a simple periodic segment over $[0,T]$. 
Clearly, $V^{--}_j$ is homotopic to $M_1 \times ... \times \partial M_j \times ... \times M_n$. Moreover, $V^{--}_{j_1} \cap ... \cap V^{--}_{j_m}$ is homotopic to $M_1 \times ... \times \partial M_{j_1} \times ... \times \partial M_{j_m} \times...\times M_n$.

If $M_i$ has a nonempty boundary, then $\partial M_i$ consists either of a finite number of curves that are homeomorphic to circles ($\chi(\partial M_i) = 0$) or it is a two-pointed set ($\chi(\partial M_i) = 2$). Suppose, without loss of generality, that for some $k$ $\chi(\partial M_i) = 0$ iff $i>k$, i.e. there are $k\geqslant 0$ one-dimensional manifolds in $\{M_i\}$. From~(\ref{sum}) we have
\begin{equation*}
\begin{aligned}
\chi(W^{--}) &= \sum_{1 \leqslant i \leqslant n} \chi(V_i^{--}) - \sum_{1 \leqslant i < j \leqslant n} \chi(V_i^{--}\cap V_j^{--}) + ... +(-1)^{n-1} \chi (V_1^{--}\cap ... \cap V_n^{--})\\
&=\sum_{1 \leqslant i \leqslant k} \chi(V_i^{--}) - \sum_{1 \leqslant i < j \leqslant k} \chi(V_i^{--}\cap V_j^{--}) + ... +(-1)^{k-1} \chi (V_1^{--}\cap ... \cap V_k^{--})\\
&=\chi(M_1 \times ... \times M_n)\cdot(2C^1_k - 2^2C^2_k + ... + 2^k (-1)^{k-1} C_k^k)=\chi(M)\cdot(1+(-1)^{k+1}).
\end{aligned}
\end{equation*}
Here we use that if $j_1 < ... < j_m \leqslant k$, then $$\chi(M_1 \times ... \times \partial M_{j_1} \times ... \times \partial M_{j_m} \times...\times M_n) = 2^m \chi(M_1 \times ... \times M_n).$$
Finally,  $\chi(W_0) = \chi(M) = \chi(M_1) ... \chi(M_n) \ne 0$, $\chi(W^{--}) = \chi(W^-_0)$. Therefore, $$\chi(W_0) - \chi(W_0^-) \ne 0$$ and we can apply~theorem~\ref{th1}.
\end{proof}
Let us now illustrate the theorem on examples. 
\begin{example}
\label{ex1main}
Consider a finite number of planar pendulums moving with viscous friction in a gravitational field (Fig.~\ref{fig1}). Let $r_i$ be the radius-vector from a fixed origin to the massive point of the $i$-th pendulum. Suppose that the pivot points of the pendulums are moving along a horizontal line in accordance with a $T$-periodic law of motion $h \colon \mathbb{R}/T\mathbb{Z} \to \mathbb{R}$, which is the same for all pendulums. Let us also assume the following: for any two pendulums the distance between their pivot points is greater than the sum of their lengths, i.e. the pendulums are disjoint; for any two pendulums there is a repelling force $F_{ij}$ acting on the massive point of the $i$-th pendulum from the $j$-th pendulum ($F_{ij}$ is parallel to $r_i - r_j$).
\begin{figure}[h!]
\centering
\def\svgwidth{400 pt}
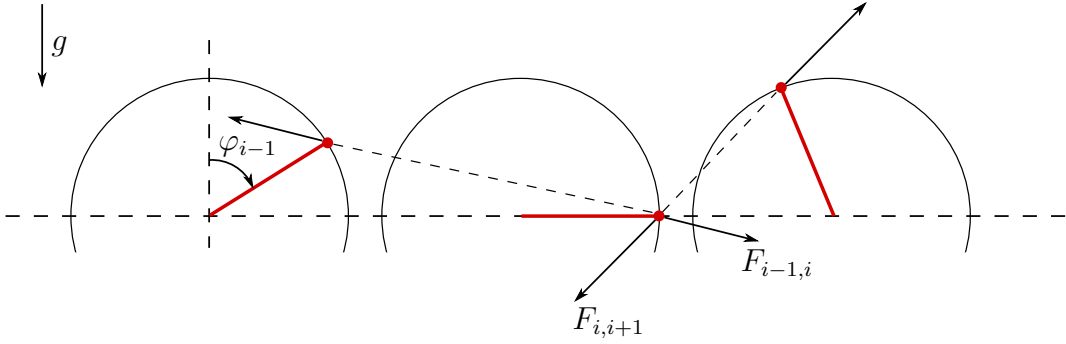
\caption{When the $i$-th pendulum is horizontal, the repelling forces acting on it are directed downward.}
\label{fig1}
\end{figure}

The equations of motion can be presented in the following form:
\begin{equation}
\label{ex1}
\begin{aligned}
&\dot\varphi_i = p_i,\\
&\dot p_i = -\frac{\ddot h}{l_i}\cos\varphi_i + \frac{g}{l_i}\sin\varphi_i - \gamma_i p_i + \frac{1}{m_i l_i}\sum\limits_{j \ne i} (F_{ij}, e_{\varphi_i}).
\end{aligned}
\end{equation}
Here $m_i$ is the mass of the $i$-th massive point; $l_i$ is the length of the $i$-th pendulum; $\gamma_i > 0$ is the viscous friction coefficient for the $i$-th pendulum (divided by the mass $m_i$); $e_{\varphi_i}$ is the following unit vector $e_{\varphi_i} = (\cos\varphi_i, -\sin\varphi_i)$; for given $\varphi_i$, $\varphi_j$, $p_i$, $p_j$, we assume that
$$
F_{ij}(\varphi_i, p_i, \varphi_j, p_j) = \frac{r_i - r_j}{|r_i - r_j|}\cdot f_{ij}(\varphi_i, \varphi_j),
$$
where $f_{ij} \geqslant 0$ is a smooth real-valued function. Then there exists a $T$-periodic solution~of~(\ref{ex1}).

Indeed, let $M_i$ be a compact manifold defined by the inequalities $-\pi/2 \leqslant \varphi_i \leqslant \pi/2$, $\chi(M_i)=1$. From~(\ref{ex1}) it follows that
\begin{equation}
\begin{aligned}
&\dot p_i \big|_{\varphi_i = \pi/2, p_i = 0} = \frac{g}{l_i} + \frac{1}{m_i l_i} \sum\limits_{j \ne i} (F_{ij}, e_{\varphi_i})\big|_{\varphi_i = \pi/2} > 0,\\
&\dot p_i \big|_{\varphi_i = -\pi/2, p_i = 0} = -\frac{g}{l_i} + \frac{1}{m_i l_i} \sum\limits_{j \ne i} (F_{ij}, e_{\varphi_i})\big|_{\varphi_i = -\pi/2} < 0.
\end{aligned}
\end{equation}
Therefore, theorem~\ref{MyMainTh} can be applied. Let us also note that along the obtained periodic solution we always have $\varphi_i \in (-\pi/2, \pi/2)$, i.e. the pendulums always remain above the horizontal line.
\end{example}
\begin{remark}
Note that if in the previous example we replace the half-circles by any compact manifolds with boundaries such that their Euler-Poincar\'{e} characteristics are non-zero, the boundaries are in a horizontal plane, the rest of the manifolds are above the horizontal plane and they are vertical at the boundaries, then, for such a system, there also exists a periodic solution. Moreover, we can consider arbitrary external periodic fields acting on the massive points instead of the periodic motion of the surfaces. If the external fields are directed downward at the horizontal plane then there exists a periodic solution (Fig.~\ref{fig2}).
\end{remark}
\begin{figure}[h!]
\centering
\def\svgwidth{320 pt}
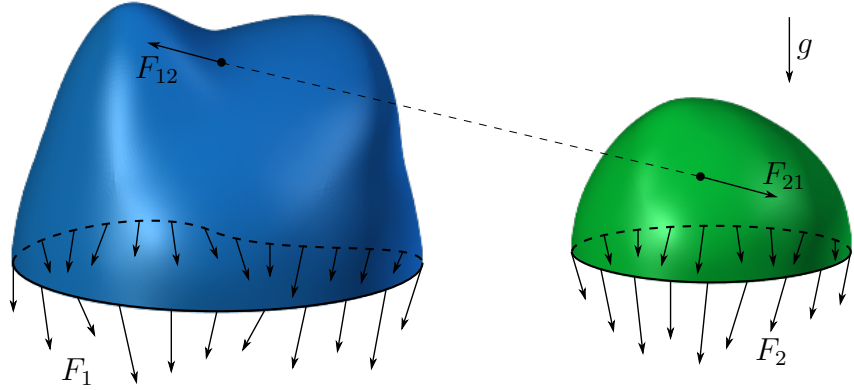
\caption{The points are repelling; the external fields $F_1$ and $F_2$ are periodic and directed downward at the boundaries.}
\label{fig2}
\end{figure}

\begin{remark}
One can also consider systems of coupled pendulums in which the pendulums attract each other. An example of such a system is presented in the figure below (Fig.~\ref{fig3}). Here we assume that the massive points of the pendulums are charged, they interact with the Coulomb interaction and the pendulums are located in a parallel electric field $E$. This system is similar to example~\ref{ex1main} and also has a periodic solution.
\begin{figure}[h!]
\centering
\def\svgwidth{400 pt}
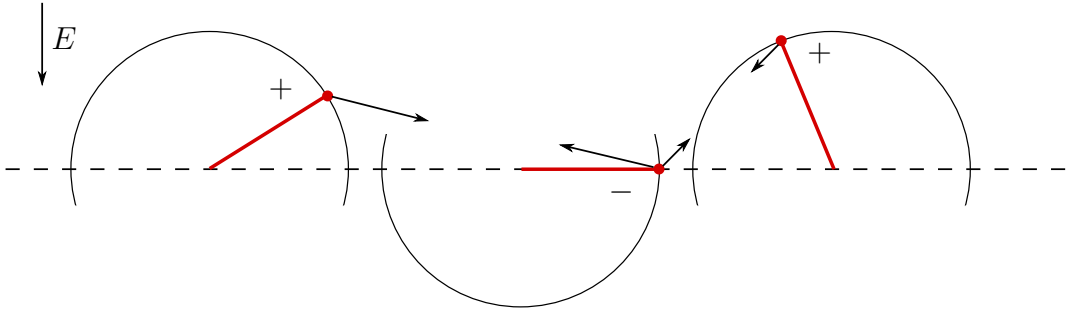
\caption{An example of the system with a periodic solution. The pendulums attracts each other and their pivot points are moving periodically along the horizontal line.}
\label{fig3}
\end{figure}
\end{remark}

\begin{example}
Consider a chain of coupled oscillators in an external periodic field. We suppose that the oscillators are located along a straight line, i.e. every oscillator has at most two neighbours. For the boundary oscillators, which have only one neighbour, we assume that the corresponding massive points are fixed. We also assume that the oscillators are moving with viscous friction and the interaction between them is described by a Morse potential (Fig.~\ref{fig4}). For simplicity, let $m_i = 1$ and $\gamma > 0$ is the same viscous friction coefficient for all oscillators. The equations of motion can be written as follows:
\begin{equation}
\label{exosc}
\ddot x_i = -\gamma \dot x_i - \frac{\partial V(x_i - x_{i-1})}{\partial x_i} - \frac{\partial V(x_{i+1} - x_i)}{\partial x_i} + F(t, x_i), \quad i=1,...,n,
\end{equation}
where
\begin{equation*}
V \colon \mathbb{R} \to \mathbb{R},\quad V(x) = \frac{1}{2}(1 - e^{-(x - \delta)})^2, 
\end{equation*}
$F \colon \mathbb{R}/T\mathbb{Z} \times \mathbb{R} \to \mathbb{R}$ and $\gamma > 0$, $\delta > 0$, $x_0 = 0$ and $x_{n+1} = (2n+1)(\delta + a)$, $a \geqslant \ln 2$ is a parameter.
\begin{figure}[h!]
\centering
\def\svgwidth{400 pt}
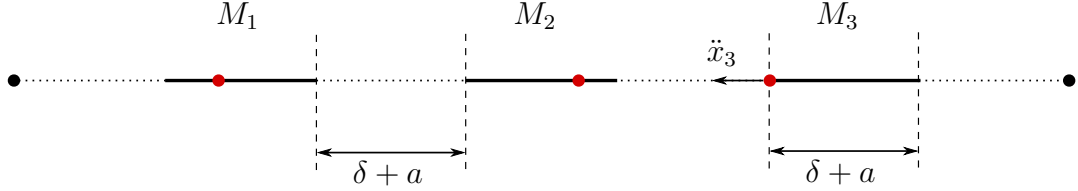
\caption{For $n = 3$ we have three compact manifolds with boundaries. The manifolds $M_i$ are distanced from each other so the oscillators attract one another; $\ddot x_i$ is directed outward at the boundary $\partial M_i$.}
\label{fig4}
\end{figure}

\noindent Let us now show that if the function $F$ satisfies the following condition
\begin{equation}
\label{excond214}
(-1)^{k+1}\cdot F(t, k(\delta + a))<0, \mbox{ for all}\quad t \in \mathbb{R}/T\mathbb{Z}, k = 1,...,2n,
\end{equation}
then the system~(\ref{exosc}) has a periodic solution. Indeed, let $M_i$ be the following compact manifold with a boundary
$$
M_i = \{ x_i \in \mathbb{R} \colon (2i-1)(\delta + a) \leqslant x_i \leqslant 2i(\delta + a) \}.
$$
Since $\left.\frac{\partial^2 V}{\partial x^2}\right|_{x = \delta + \ln 2} = 0$ and $\frac{\partial^2 V}{\partial x^2} < 0$ for $x > \delta + \ln 2$, then 
$$
\left.\frac{\partial V}{\partial x}\right|_{x = x_2} \leqslant \left.\frac{\partial V}{\partial x}\right|_{x = x_1} \quad \mbox{if} \quad \delta + \ln 2 \leqslant x_1 \leqslant x_2.
$$
Finally, taking into account~(\ref{excond214}), we obtain that $\ddot x_i$ is directed outward to $M_i$ at the boundary $\partial M_i$ provided $\dot x_i = 0$ and~\ref{MyMainTh} can be applied.
\end{example}

\section{Conclusion}
\label{sec3}



Taking into account that the presented approach for studying forced oscillations is relatively simple, and, at the same time, it allows one to use a wide range of `building blocks' and types of forces to construct a physical system that has a periodic solution, we hope that our result will be useful in applications.



\bibliographystyle{elsarticle-num}


\bibliography{sample}

\end{document}